\documentclass{article}
\usepackage{amsmath}
\usepackage{amssymb}

\newtheorem{theorem}{Theorem}

\newtheorem{corollary}[theorem]{Corollary}

\newtheorem{lemma}[theorem]{Lemma}

\newenvironment{proof}[1][Proof]{\textbf{#1.} }{\ \rule{0.5em}{0.5em}}
\input{tcilatex}
\begin{document}

\title{Relationships Between Characteristic Path Length,\\
Efficiency, Clustering Coefficients, and Density}
\author{Alexander Strang \\
%EndAName
Case Western Reserve University\\
Cleveland, OH 44106 \and Oliver Haynes \\
%EndAName
Rochester Institute of Technology\\
Rochester, NY 14623 \and Nathan D. Cahill \\
%EndAName
Rochester Institute of Technology\\
Rochester, NY 14623 \and Darren A. Narayan \\
%EndAName
Rochester Institute of Technology\\
Rochester, NY 14623-5604}
\date{}
\maketitle

\begin{abstract}
Graph theoretic properties such as the clustering coefficient,
characteristic (or average) path length, global and local efficiency,
provide valuable information regarding the structure of a graph. These four
properties have applications to biological and social networks and have
dominated much of the the literature in these fields. While much work has
done in applied settings, there has yet to be a mathematical comparison of
these metrics from a theoretical standpoint. Motivated by networks appearing
in neuroscience, we show in this paper that these properties can be linked
together using a single property - graph density. In a recent paper
appearing in this journal we presented data from studies in neuroscience
that suggested that $E_{loc}(G)\approx \frac{1}{2}\left( 1+CC(G)\right) $.
In this paper, we verify that this relationship holds in general and
establish additional connections between graph metrics.
\end{abstract}

\section{Introduction}

Graph theory provides an abundance of valuable tools for analyzing social
and biological networks. There are many well known distance metrics which
are used to analyze networks including diameter, density, characteristic
path length, clustering coefficient, global and local efficiency which have
been extensively studied \cite{BS},.\cite{EVCN}, \cite{Ek}, \cite{Fukami}, 
\cite{Lancichinetti}, \cite{LM}, \cite{Li}, \cite{Nascimento}, \cite{Shang}, 
\cite{Sporns}, \cite{WS}, \cite{Zhang}, and \cite{Vargas}. Motivated by real
world data we present asymptotic linear relationships between the
characteristic path length, global efficiency, and graph density and also
between the clustering coefficient and local efficiency. In the current
literature these properties are often presented as independent metrics,
however we show in this paper that they are inextricably linked through a
single property - graph density.

The distance between two vertices in a network is simply the number of edges
in a shortest path connecting them. The characteristic path length (or
average path length) is the average of all of the distances over all pairs
of vertices in a network. The global efficiency is the average of all of the
reciprocals of the non-zero distances in a network \cite{LM}. There are also
two well known local properties for graphs which are based at the vertex
level.

In a network the \textit{open neighborhood} of a vertex is obtained by first
identifying all of the nodes with a direct link to the starting vertex.
These will be referred to as its "neighbors". Then\ among these nodes take
all of the edges appear in the original graph. The combination of the
neighbors and the edges forms the open neighborhood. The \textit{closed
neighborhood} of a node is the same idea as the open neighborhood, but we
include the starting vertex and edges connected to it. The two types of
neighborhoods are illustrated in Figure 1.

\FRAME{dtbpFU}{5.4379in}{2.0799in}{0pt}{\Qcb{Figure 1. (i) Network (ii) open
neighborhood of $x$ (iii) closed neighborhood of $x$}}{}{neighborhoods.jpg}{%
\special{language "Scientific Word";type "GRAPHIC";maintain-aspect-ratio
TRUE;display "USEDEF";valid_file "F";width 5.4379in;height 2.0799in;depth
0pt;original-width 7.2151in;original-height 2.7363in;cropleft "0";croptop
"1";cropright "1";cropbottom "0";filename
'../../ComplexNetworks/Neighborhoods.jpg';file-properties "XNPEU";}}

In this paper we show other relations between graph metrics and support them
with precise mathematical justification. Motivated by networks appearing in
neuroscience, we present relationships among properties involving shortest
paths: (i) Local efficiency has a precise linear relationship with the
clustering coefficient for the closed neighborhood version. (ii) Local
efficiency has an asymptotic linear relationship with the clustering
coefficient for the open neighborhood version, and (iii) global efficiency
has an asymptotic linear relationship with the characteristic path length.
Furthermore we show that these linear relationships are linked through a
single property - \textit{graph density}. The density of a graph is simply
the ratio of the number of edges in a graph to the number of possible edges.

\section{Preliminaries}

We begin by reviewing some well known properties of graphs. A graph $G$ is
comprised of a set of vertices and a set of (undirected) edges where an edge
joins two vertices. Unless otherwise stated we will use $n$ to denote the
number of vertices in a graph and $m$ to denote the number of edges in a
graph. The maximum number of edges in a graph $G$ with $n$ vertices is $%
\binom{n}{2}$. Given a graph $G$ with $m$ edges the density $D(G)$ equals $%
\frac{2m}{n(n-1)}$. Let $G$ be a graph with $n$ vertices and let $d(i,j)$
represent the \textit{distance} (the number of edges in a shortest path
between vertices $i$ and $j$ in $G$). If there is no path connecting $i$ and 
$j$, then $d(i,j)=\infty $. The \textit{diameter} of $G$ is denoted $diam(G)$
and equals $\max_{i,j}d(i,j)$. The \textit{characteristic path length} $L$
is the average distance over all pairs of vertices, $L=\frac{1}{n(n-1)}%
\dsum\nolimits_{i\neq j}d(i,j)$. For each vertex $i$, let $G_{i}$ be the
subgraph of $G$ which are induced by the neighbors of $i$, and let $%
G_{i}^{\prime }$ be the subgraph of $G$ which are induced by vertex $i$ and
the neighbors of $i$. In 2001, Latora and Marchiori \cite{LM} defined the 
\textit{global efficiency of a graph }to be\textit{\ }$E_{glob}(G)=\frac{1}{%
n(n-1)}\dsum\nolimits_{i\neq j}\frac{1}{d(i,j)}$. \ They also defined a
local version of efficiency. The \textit{local efficiency} is defined to be $%
E_{loc}(G)=\frac{1}{n}\dsum\nolimits_{i\in G}E_{glob}\left( G_{i}\right) $
and is the average of the global efficiencies of the subgraphs $G_{i}$ \cite%
{LM}. The\textit{\ clustering coefficient} was defined by Watts and Strogatz 
\cite{WS} to be $CC(G)=\frac{1}{n}\dsum\nolimits_{i}\frac{\left\vert
E(G_{i})\right\vert }{\binom{\left\vert V(G_{i})\right\vert }{2}}$ where $%
\left\vert V(G_{i})\right\vert $ is size of the vertex set of the graph $%
G_{i}$, and $\left\vert E(G_{i})\right\vert $ is size of the edge set of the
graph $G_{i}$. Closed variants of the clustering coefficient and local
efficiency can be defined by $\overline{CC}(G^{\prime })$ $=\frac{1}{n}%
\dsum\nolimits_{i}\frac{\left\vert E(G_{i}^{\prime })\right\vert }{\binom{%
V(G_{i}^{\prime })}{2}}$ and $\overline{E_{loc}}(G)=\frac{1}{n}%
\dsum\nolimits_{i\in G}E_{glob}\left( G_{i}^{\prime }\right) $.

While the properties have been well studied, more precise attention needs to
be given to the relationships between them. Analysis of biological networks,
fMRI data, and simulated (benchmark) networks that the following
inequalities suggest that $E_{loc}(G)\leq \frac{1}{2}(1+CC(G))$, $\overline{%
E_{loc}}(G)=\frac{1}{2}(1+\overline{CC}(G))$, and $\frac{1}{2}(3-L(G))\leq
E_{glob}(G)\leq \frac{1}{2}(1+D(G))$. In this paper, we prove that these
relationships hold in general and show that the inequalities approach
equality as the density of the graph increases.

\section{Connections between graph properties}

\subsection{Motivation from biological networks}

Results in this section were motivated by a research study of McCarthy,
Benuskova, and Franz where the clustering coefficient and local efficiency
properties were applied to functional MRI data from subjects with
posterior-anterior shift in aging (PASA) \cite{McCarthy1} and \cite%
{McCarthy2}. Figure 2 appears in \cite{McCarthy1} and shows that the
clustering coefficient and local efficiency data for both task based resting
state functional networks approaches the line $E_{loc}(G)$ $=$ $\frac{1}{2}%
(1+CC(G))$ from below, which we will show is consistent with our theoretical
findings. We will show later that Theorem \ref{2} gives a decent
approximation of the equation of this line segment when $0.2\leq CC(G)\leq 1$
which can be obtained using the points $(0.2,0.6)$ and $(1,1)$. The result
is $E_{loc}(G)$ $\approx $ $\frac{1}{2}(1+CC(G))$ and we show later in the
paper that this approximation improves as the density and clustering
coefficient approach $1$.

\FRAME{dtbpFw}{4.7444in}{3.1185in}{0pt}{}{}{mccarthy.jpg}{\special{language
"Scientific Word";type "GRAPHIC";maintain-aspect-ratio TRUE;display
"USEDEF";valid_file "F";width 4.7444in;height 3.1185in;depth
0pt;original-width 7.2566in;original-height 4.7573in;cropleft "0";croptop
"1";cropright "1";cropbottom "0";filename
'../../ComplexNetworks/McCarthy.jpg';file-properties "XNPEU";}}

\begin{center}
Figure 2. Comparison of $E_{loc\text{ }}(G)$ and $CC(G)$ for subjects with
PASA. \cite{McCarthy1}.

\bigskip
\end{center}

\subsection{Motivation from analysis of functional MRI data}

In this section we investigate connections between the two global
properties, characteristic path length $L(G)$, and global efficiency $%
E_{glob}$, motivated by a study of resting state functional MRI scans
conducted by the Rochester Center for Brain Imaging at the University of
Rochester in 2013. The data from the scans reflected correlations in blood
oxygenated level dependent (BOLD) signals between pairs of 92 selected
regions of the brain. The Pearson correlations were all between $1$ (perfect
correlation) and $-1$ (perfect anti-correlation). The correlations were then
binarized with a threshold of $0$ to form an adjacency matrix for each scan.
Analysis of the 25 resting state functional MRI data sets resulted in the
following averages: $D(G)=0.483$, $L(G)=1.518$, and $E_{glob}(G)=0.741$. We
note that $D(G)\approx 0.5$, $L(G)\approx 1.5$, and $E_{glob}(G)\approx 0.75$%
. Working towards a linear model we note that when $D(G)=0$, $%
L(G)=E_{glob}(G)=0$, and when $D(G)=1$, $L(G)=E_{glob}(G)=1$. We obtained a
linear regression showing connections between the characteristic path length
and the global efficiency to the density, $L(G)=2-D(G)$ and $E_{glob}(G)=%
\frac{1}{2}\left( 1+D(G)\right) $.

For another analysis we analyzed data from another functional MRI study
conducted at the Rochester Center for Brain Imaging at the\ University of
Rochester. The study involved 16 subjects who were asked to either view or
pantomime various tools over the course of 8 scans (128 total scans). The
data from the scans reflected correlations in blood oxygenated level
dependent (BOLD) signals between pairs of 11 selected regions of the brain
known to respond to tool viewing and /or pantomiming. The Pearson
correlations were all between $1$ (perfect correlation) and $-1$ (perfect
anti-correlation). The correlations were then binarized with a threshold of $%
0$ to form an adjacency matrix for each scan. Analysis of the 128 graphs
yielded some interesting findings:\ $E_{loc}(G)\leq \frac{1}{2}(1+CC(G))$
and $E_{glob}(G)\leq \frac{1}{2}(1+D(G))$. We also found that for 127 of the
graphs, $E_{glob}(G)\geq $ $\frac{1}{2}(3-L(G))$ (The remaining graph was
disconnected resulting in an infinite value for $L)$. The maximum deviation
between $E_{glob}(G)$ and $\frac{1}{2}(1+D(G))$ was $0.0909$ with $%
E_{glob}(G)$ $=\frac{1}{2}(1+D(G))$ when the density of the graph was
greater than $0.7818$. \ In addition, for the 127 connected graphs the
maximum deviation between $E_{glob}(G)$ and $\frac{1}{2}(3-L(G))$ was $0.05$%
. We also found that $E_{glob}(G)$ $=\frac{1}{2}(3-L(G))$ when the density
was greater than $0.7818$. Furthemore, we found that $\left\vert \overline{%
E_{loc}}(G)-\frac{1}{2}\left( 1+\overline{CC}(G)\right) \right\vert \leq
0.00005$ which matches with our Theorem \ref{1} other than a slight
deviation due to rounding.

We will show in subsection 3.4 that these approximations approach equality
as the density of the graphs approaches $1$.

\subsection{Motivation from benchmark simulations}

To generate Figures 3, 4, and 5 we constructed benchmark graphs with 128
vertices having degree and community size distributions governed by power
laws with exponents $2$ and $1$, respectively, where each vertex shares a
fraction of $0.8$ of its edges with other vertices in its community. The
Lancichinetti-Fortunato-Radicchi (LFR) benchmark graphs \cite{Lancichinetti}
enable the user to define the desired average vertex degree. We generated
LFR benchmark graphs for average vertex degrees ranging from $4$ to $64$ (by 
$1$); for each average vertex degree, we generated $30$ realizations of
benchmark graphs. For each of the $30\times 61=1830$ benchmark graphs we
generated, we computed $E_{glob}(G)$, $E_{loc}(G)$, $D(G)$, and $L(G)$. Our
analysis is that for Figure 3, $E_{loc}(G)\leq \frac{1}{2}(1+CC(G))$; Figure
4, $E_{glob}(G)\leq \frac{1}{2}(1+D(G))$; and Figure 5, $E_{glob}(G)\geq $ $%
\frac{1}{2}(3-L(G))$, and all of these inequalities asymptotically approach
equality as the density approaches $1$.

\begin{tabular}{ll}
\FRAME{itbpFU}{2.8573in}{2.188in}{0in}{\Qcb{Figure 3. $E_{loc}(G)$ compared
with $CC(G)$}}{}{Figure}{\special{language "Scientific Word";type
"GRAPHIC";maintain-aspect-ratio TRUE;display "USEDEF";valid_file "T";width
2.8573in;height 2.188in;depth 0in;original-width 5.5486in;original-height
4.2359in;cropleft "0";croptop "1";cropright "1";cropbottom "0";tempfilename
'OWOGZT00.wmf';tempfile-properties "XPR";}} & \FRAME{itbpFU}{2.7406in}{%
2.1811in}{0in}{\Qcb{Figure 4. $E_{glob}(G)$ compared with $D(G)$}}{}{Figure}{%
\special{language "Scientific Word";type "GRAPHIC";maintain-aspect-ratio
TRUE;display "USEDEF";valid_file "T";width 2.7406in;height 2.1811in;depth
0in;original-width 5.3195in;original-height 4.222in;cropleft "0";croptop
"1";cropright "1";cropbottom "0";tempfilename
'OWOGZT01.wmf';tempfile-properties "XPR";}}%
\end{tabular}

\FRAME{dtbpFU}{2.8461in}{2.188in}{0pt}{\Qcb{Figure 5. $E_{loc}(G)$ compared
with $L(G)$}}{}{Figure}{\special{language "Scientific Word";type
"GRAPHIC";maintain-aspect-ratio TRUE;display "USEDEF";valid_file "T";width
2.8461in;height 2.188in;depth 0pt;original-width 5.5279in;original-height
4.2359in;cropleft "0";croptop "1";cropright "1";cropbottom "0";tempfilename
'OWOGZT02.wmf';tempfile-properties "XPR";}}

\subsection{General results}

Next we establish relationships between the local efficiency and the
clustering coefficient.

\begin{lemma}
Let $v$ be a vertex in $G$ and let $G_{v}$ be the subgraph induced by the
vertices in the closed neighborhood of $v$. Then $diam(G_{v})\leq 2$.
\end{lemma}

\begin{proof}
By definition of $G_{v}$, $v$ is adjacent to all other vertices. For any
pair of vertices $i$ and $j$ not equal to $v$ they are either adjacent or
connected by a path through $v$.
\end{proof}

Given a graph $G$, let $A(G)$ be the adjacency matrix and let $Eff(G)$ be
the efficiency matrix where $Eff_{i,j}(G)=\frac{1}{d(i,j)}$ for all $i\neq j$%
. We note that if $(i,j)$ is an edge of $G$ then $A_{i,j}(G)=Eff_{i,j}(G)=1$
and if $(i,j)$ is not an edge of $G$ then $Eff_{i,j}(G)-A_{i,j}(G)=\frac{1}{%
d(i,j)}$.

\begin{theorem}
\label{1}For any graph $G$, $\overline{E_{loc}}(G)=\frac{1}{2}(1+\overline{CC%
}(G))$.
\end{theorem}

\begin{proof}
Let $v$ vertex in $G$ and let $G_{v}^{\prime }$ be the subgraph induced by
the vertices in the closed neighborhood of $v$. Let $n(G_{v}^{\prime })$ be
the number of vertices in $G_{v}^{\prime }$. Then $Eff_{i,j}(G)-A_{i,j}(G)=0$
or $\ Eff_{i,j}(G)-A_{i,j}(G)=\frac{1}{2}$. \medskip

First we have $\overline{CC}(G_{v}^{\prime })=\dsum\limits_{\left(
i,j\right) \in E(G)}\frac{A_{i,j}(G_{v}^{\prime })}{n(G_{v}^{\prime
})(n(G_{v}^{\prime })-1)}$.

Next we have,

$\overline{E_{loc}}(G_{v}^{\prime })=\dsum\limits_{\left( i,j\right) }\frac{1%
}{n(G_{v}^{\prime })(n(G_{v}^{\prime })-1)}\left( Eff_{i,j}(G_{v}^{\prime
})\medskip \right) $

$\medskip =\frac{1}{n(G_{v}^{\prime })(n(G_{v}^{\prime })-1)}\left(
\dsum\limits_{\left( i,j\right) \in E(G_{v}^{\prime
})}1+\dsum\limits_{\left( i,j\right) \notin E(G_{v}^{\prime })}\frac{1}{2}%
\right) $

$=\frac{1}{2}\left( 1+\overline{CC}(G_{v}^{\prime })\right) .\medskip $

Summing over all vertices $v$ gives $\overline{E_{loc}}(G)=\frac{1}{2}\left(
1+\overline{CC}(G)\right) $.\medskip

In the next theorem we show a relationship between the open neighborhood
versions of local efficiency and the clustering coefficient.
\end{proof}

\begin{theorem}
\label{3}For any graph $G$, $E_{loc}(G)\leq \frac{1}{2}(1+CC(G))$.
\end{theorem}

\begin{proof}
Let $v$ be a vertex of $G$. $\ $Then $E_{glob}(G_{v})\leq
\dsum\limits_{\left( i,j\right) \in Eff(G_{v})}1+\dsum\limits_{\left(
i,j\right) \notin Eff(G_{v})}\frac{1}{2}\medskip \medskip =\frac{1}{2}\left(
1+CC(G_{v})\right) .\medskip $

Summing over all vertices $v$ gives $E_{loc}(G)\leq \frac{1}{2}\left(
1+CC(G)\right) $.\medskip
\end{proof}

In general the bound cannot be improved since there is equality when $G$ is
a complete graph.

We note in the next theorem in graphs where the distance between most
vertices is less than or equal to $2$ there is a near linear approximation
between $E_{loc}(G)$ and $CC(G)$.

\begin{theorem}
\label{2}As the fraction of vertex pairs $u,v$ with $d(u,v)\leq 2$
approaches $1$, then $E_{loc}(G)$ approaches $\frac{1}{2}\left(
1+CC(G)\right) $.
\end{theorem}

\begin{proof}
The proof is similar to the approach found in Theorem \ref{1}, noting that
the deviation between $E_{loc}(G)$ and $\frac{1}{2}\left( 1+CC(G)\right) $
is directly related to the number of pairs of vertices that have a distance
of more than $2$. As this quantity decreases our approximation becomes
closer.
\end{proof}

In our first two lemmas, we investigate bounds between the different
properties.

\begin{lemma}
\label{6}For any graph $G$, $E_{glob}(G)\leq \frac{1}{2}(1+D(G))$.
\end{lemma}

\begin{proof}
We can express $E_{glob}(G)$ as $\frac{1}{n(n-1)}\left( m+\frac{1}{2}\alpha
+\epsilon \beta \right) $ where $\alpha $ is the number of pairs of vertices
which are separated by a distance of exactly $2$, and $\beta $ is the number
of pairs of vertices whose distance is greater than $2$ and hence has an
efficiency of $\epsilon $ $<\frac{1}{2}$. Then $E_{glob}(G)$ $=$ $\frac{1}{%
n(n-1)}\left( m+\frac{1}{2}\alpha +\epsilon \beta \right) $ $\leq $ $\frac{1%
}{n(n-1)}\left( m+\frac{1}{2}(\alpha +\beta )\right) $ $=$ $\frac{1}{n(n-1)}%
\left( m+\frac{n(n-1)-m}{2}\right) $ $=$ $\frac{1}{n(n-1)}\left( \frac{m}{2}+%
\frac{n(n-1)}{2}\right) $ $=$ $\frac{1}{2}\left( 1+\frac{m}{n(n-1)}\right) $ 
$=$ $\frac{1}{2}(1+D(G))$.
\end{proof}

\begin{lemma}
\label{7}For any graph $G$, $L(G)\geq 2-D(G)$.
\end{lemma}

\begin{proof}
$L=\frac{m+2\alpha +\epsilon \beta }{n(n-1)}\geq \frac{m+2\left( \alpha
+\beta \right) }{n(n-1)}=\frac{2\left( m+\alpha +\beta \right) }{n(n-1)}-%
\frac{m}{n(n-1)}=2-D$.
\end{proof}

We note that the bound in the previous lemma is tight for the case where $G$
is a complete graph.

The combination of Lemmas \ref{6} and \ref{7} yields the following theorem.

\bigskip

\begin{theorem}
\label{8}For any graph $G$ , $E_{glob}(G)\geq \frac{1}{2}(3-L(G))$.
\end{theorem}

\begin{proof}
Following from above we note that $E_{glob}(G)=\frac{1}{n(n-1)}\left( m+%
\frac{1}{2}\alpha +\epsilon \beta \right) $ and $L=\frac{m+2\alpha +\epsilon
\beta }{n(n-1)}$.
\end{proof}

Then $\frac{1}{n(n-1)}\left( m+\frac{1}{2}\alpha +\epsilon \beta \right)
\geq \frac{1}{2}$\bigskip $\left( 3-\frac{m+2\alpha +\epsilon \beta }{n(n-1)}%
\right) $

$\Leftrightarrow \frac{1}{n(n-1)}\left( m+\frac{1}{2}\alpha +\epsilon \beta
\right) \geq \frac{3}{2}-\frac{1}{2}\bigskip \frac{m+2\alpha +\epsilon \beta 
}{n(n-1)}$

\bigskip $\Leftrightarrow \frac{1}{n(n-1)}\left( m+\frac{1}{2}\alpha
+\epsilon \beta \right) +\frac{1}{2}\bigskip \frac{m+2\alpha +\epsilon \beta 
}{n(n-1)}\geq \frac{3}{2}$

$\Leftrightarrow \frac{1}{n(n-1)}\left( \left( m+\frac{1}{2}\alpha +\epsilon
\beta \right) +\frac{1}{2}\left( m+2\alpha +\epsilon \beta \right) \right)
\geq \frac{3}{2}$

\bigskip

$\Leftrightarrow \frac{1}{n(n-1)}\left( \left( m+\frac{1}{2}\alpha +\epsilon
\beta \right) +\frac{1}{2}\left( m+2\alpha +\epsilon \beta \right) \right)
\geq \frac{3}{2}$

$\bigskip $

Note that $\frac{1}{n(n-1)}\left( \left( m+\frac{1}{2}\alpha +\epsilon \beta
\right) +\frac{1}{2}\left( m+2\alpha +\epsilon \beta \right) \right)
\medskip $

$=\frac{1}{n(n-1)}\allowbreak \left( \frac{3}{2}m+\frac{3}{2}\alpha +\frac{3%
}{2}\beta \epsilon \right) \medskip $

$=\frac{1}{n(n-1)}\frac{3}{2}\left( m+\alpha +\epsilon \beta \right)
\medskip $

$=\frac{1}{n(n-1)}\frac{3}{2}\left( m+\alpha +\beta +(1-\epsilon )\beta
\right) \medskip $

$=\frac{3}{2}\left( \frac{m+\alpha +\beta }{n(n-1)}+\frac{(1-\epsilon )\beta 
}{n(n-1)}\right) \medskip \medskip $

$=\frac{3}{2}\left( 1+\frac{(1-\epsilon )\beta }{n(n-1)}\right) \geq \frac{3%
}{2}$.

\begin{corollary}
The bounds in Lemmas \ref{6} and \ref{7} and Theorem \ref{8} are tight for
the case where $G$ is a graph where $diam(G)\leq 2$.
\end{corollary}

\begin{proof}
The proofs for these cases follow by considering $\beta =0$.
\end{proof}

\begin{theorem}
\label{9}As the fraction of vertex pairs $u,v$ with $d(u,v)\leq 2$
approaches $1$, then $E_{glob}(G)\ $approaches $\frac{1}{2}(3-L(G))$.
\end{theorem}

\begin{proof}
The proof is obtained by combining Lemmas \ref{6} and \ref{7} and Theorem %
\ref{8} and noting that the deviation between $E_{glob}(G)$ and $\frac{1}{2}%
(3-L(G))$ is tied to the number of pairs of vertices that have a distance of
more than $2$. As this quantity decreases our approximation becomes closer.
\end{proof}

\section{Discussion}

We have established a relationship between the characteristic path length
and global efficiency. Likewise we showed a similar link between the local
efficiency and the clustering coefficient (both open and closed versions).
In both of these cases we showed that the relationships converge as the
density of the graph approaches $1$.

We found that relationships between local efficiency and the clustering
coefficient (open versions) become prevalent when the graph has a density
around $0.2$. This was shown to be consistent with real world data findings
such as the study by McCarthy, Benuskova, and Franz \cite{McCarthy2}. It
would be interesting to conduct an analysis using more real world data to
see precisely how the relationships between graph properties behave in
networks with low density.

We note that all of these graph properties are dependent on the structure of
the network. In particular the distances between various pairs of nodes.
More precisely a network's structure is dependent upon the "distance
distribution" that is the percentage pairs of nodes that are separated by
distance $d$. This presents a problem of a probabilistic nature which could
be explored using techniques from random graphs such as Erd\"{o}s-R\'{e}yni
models \cite{ErdosRenyi}.

\bigskip

{\large \noindent Acknowledgements\medskip }

Research was supported by a National Science Foundation Research Experiences
for Undergraduates Grant \#1358583. Darren Narayan was also supported by a
National Science Foundation CCLI Grant \#1019532. The authors would like to
thank Bradford Z. Mahon and Frank Garcea of the Rochester Center for Brain
Imaging at the University of Rochester and Dr. Jeffery Bazarian and \ Kian
Merchant-Borna of the Department of Emergency Medicine for providing
functional MRI data.


\begin{thebibliography}{99}
\bibitem{Brandes} U. Brandes, S. P. Borgatti, L. C. Freeman, Maintaining the
duality of closeness and betweenness centrality. Social Networks, Volume 44,
January 2016, Pages 153--159.

\bibitem{BS} E. Bullmore and O. Sporns, Complex brain networks:\ graph
theoretical analysis of structural and functional systems, Nature \textbf{10}
(2009) 186-196.

\bibitem{EVCN} B. Ek, C. VerSchneider, N. Cahill, and D. A. Narayan, A
Comprehensive Comparison of Graph Theory Metrics for Social Networks, Social
Network Analysis and Mining, Vol. 5, No. 1, (2015), 1-7.

\bibitem{Ek} B. Ek, Bryan, C. VerSchneider, D. A. Narayan, Global efficiency
of graphs. AKCE Int. J. Graphs Comb. 12 (2015), no. 1, 1--13.

\bibitem{ErdosRenyi} P. Erd\H{o}s, A. R\'{e}nyi (1959). "On Random Graphs.
I" (PDF). Publicationes Mathematicae. 6: 290--297.

\bibitem{F} L. Freeman, A set of measures of centrality based upon
betweenness, Sociometry \textbf{40}: (1977) 35--41.

\bibitem{Freeman2} L. C. Freeman, S. P. Borgatti, and D. R. White,
Centrality in valued graphs: A measure of betweenness based on network flow.
Social Networks, 13(2): (1991) 141-154.

\bibitem{Fukami} T. Fukami and N. Takahashi, New classes of clustering
coefficient locally maximizing graphs. Discrete Appl. Math. 162 (2014),
202--213.

\bibitem{G} M. Guye, G. Bettus, F. Bartolomei, and P. J. Cozzone, Graph
theoretical analysis of structural and functional connectivity of MRI in
normal pathological brain networks, Magn. Reson. Mater. Phy. \textbf{23}
(2010), 409-421.

\bibitem{Lancichinetti} A. Lancichinetti, S. Fortunato, and F. Radicchi,
Benchmark graphs for testing community detection algorithms, Phys Rev. E,
78, 046110, (2008).

\bibitem{LM} V. Latora and M. Marchiori, Efficient Behavior of Small-World
Networks, Physical Review Letters E, Vol. 87, No. 19, (2001).

\bibitem{Li} Y. Li, Y. Shang, Y.Yang, Clustering coefficients of large
networks. Inform. Sci. 382/383 (2017), 350--358.

\bibitem{McCarthy1} P. McCarthy, Ph.D. Thesis, University of Otago,
Functional network analysis of aging and Alzeheimer's Disease, 2014.

\bibitem{McCarthy2} P. McCarthy, L. Benuskova, and E. Franz, The age-related
posterior-anterior shift as revealed by voxelwise analysis of functional
brain networks, Aging Neurosci., (2014) \TEXTsymbol{\vert} doi:
10.3389/fnagi.2014.00301.

\bibitem{Nascimento} Mari\'{a} C. V. Nascimento Community detection in
networks via a spectral heuristic based on the clustering coefficient.
Discrete Appl. Math. 176 (2014), 89--99.

\bibitem{Gr} A. S. Pandit, P. Expert, R. Lambiotte, V. Bonnelle, R. Leech,
F. Turkheimer, and D. J. Sharp. Traumatic brain injury impairs small-world
topology, Neurology (2013) \textbf{80}, 1826--1833.

\bibitem{Shang} Y. Shang, Distinct clustering and characteristic path
lengths in dynamic small-world networks with identical limit degree
distribution, J. Stat. Phys. 149 (2012), no. 3, 505--518.

\bibitem{Sporns} O. Sporns, Networks of the Brain, MIT Press, (2011).

\bibitem{Vargas} R. Vargas, F. Garcea, B. Mahon, and D. A. Narayan, Refining
the clustering coefficient for analysis of social and neural network data,
Soc. Netw. Anal. Min. (2016) 6:49, DOI 10.1007/s13278-016-0361-x

\bibitem{WS} D. Watts and S. Strogatz, Collective dynamics of `small world
networks', Nature 393, 440-442 (1998)

\bibitem{West} D. B. West, Introduction to Graph Theory (Second Edition),
Prentice-Hall (2001).

\bibitem{White} D. R. White and S. P. Borgatti. Betweenness centrality
measures for directed graphs. Social Networks, 16(4): 335-346.

\bibitem{Zhang} T. Zhang, B. Fang, X. Liang, A novel measure to identify
influential nodes in complex networks based on network global efficiency.
Modern Phys. Lett. B 29 (2015), no. 28
\end{thebibliography}
\end{document}